\documentclass{amsart}

 \usepackage{amsmath,amssymb,amscd,amsthm,amsfonts}

\usepackage{bbm}
\usepackage{hyperref}

\usepackage{enumerate}
\usepackage{mathtools}
\usepackage{stmaryrd}
\usepackage{sidenotes}
\usepackage{xcolor}

\newtheorem{thm}{Theorem}[section]

\newtheorem{prop}[thm]{Proposition}
 
\theoremstyle{definition}
\newtheorem{defn}[thm]{Definition} 
\theoremstyle{remark}

\numberwithin{equation}{section}

\newcommand{\spacedvert}{\,\vert\,}

\newcommand{\fO}{\mathfrak{O}}

\newcommand{\fa}{\mathfrak{a}}
\newcommand{\fb}{\mathfrak{b}}
\newcommand{\fc}{\mathfrak{c}}

\newcommand{\fg}{\mathfrak{g}}

\newcommand{\fm}{\mathfrak{m}}
\newcommand{\fn}{\mathfrak{n}}
\newcommand{\fp}{\mathfrak{p}} 
\newcommand{\fz}{\mathfrak{z}}

\newcommand{\CC}{\mathbb{C}}
\newcommand{\FF}{\mathbb{F}}

\newcommand{\QQ}{\mathbb{Q}}
\newcommand{\RR}{\mathbb{R}}

\newcommand{\bC}{\mathbf{C}}

\DeclareMathOperator{\spanrow}{\text{span\:\!row}}
\DeclareMathOperator{\spancol}{\text{span\:\!col}}

\DeclareMathOperator{\Stab}{\mathrm{Stab}}

\DeclareMathOperator{\GL}{\mathrm{GL}}
\DeclareMathOperator{\gl}{\mathfrak{gl}}

\DeclareMathOperator{\tr}{\mathrm{tr}}
\DeclareMathOperator{\Ad}{\mathrm{Ad}}

\DeclareMathOperator{\Ind}{\mathrm{Ind}}
\DeclareMathOperator{\uInd}{\mathrm{u\text{-}Ind}}

\DeclareMathOperator{\Mat}{\mathrm{Mat}}

\newcommand{\isom}{\cong}

\DeclareMathOperator{\rank}{\mathrm{rank}}

\newcommand{\der}{\mathrm{der}}

\newcommand{\form}[2]{\langle{#1},{#2}\rangle}

\def\smatc[#1,#2,#3,#4]{\ensuremath\bigl( \begin{smallmatrix}
#1&#2\\ #3&#4
\end{smallmatrix} \bigr)}
\def\bmatc[#1,#2,#3,#4]{\ensuremath
  \begin{pmatrix}
    #1&#2\\ #3&#4
  \end{pmatrix}}

\newtheorem{lem}[thm]{Lemma}
\newcommand{\fu}{{\mathfrak{u}}}

\newcommand{\BF}{{\mathbb{F}}}
\DeclareMathOperator{\codim}{{\mathrm{codim}}}
\newcommand{\BN}{{\mathbb{N}}}

\def\forma[#1,#2]{\ensuremath\langle #1 , #2 \rangle}

\title{Associated Representations of finite pattern groups}

\author{Chufeng Nien}
\address{Key Laboratory of Computing and Stochastic Mathematics (Ministry of Education), School of Mathematics and Statistics, Hunan Normal University, Changsha, Hunan 410081,P. R. China }
\email{nienpig@hotmail.com}

\author{Chenyan Wu}
\address{School of Mathematics and Statistics,
the University of Melbourne, Victoria 3010, Australia}
\email{chenyan.wu@unimelb.edu.au}
\thanks{} 

\date{\today}
\keywords{Coadjoint orbits, Unitriangular groups, Unipotent groups, Pattern groups, Higman Conjecture, Isaacs Conjecture, Lehrer Conjecture}

\begin{document}
\begin{abstract}
In this paper, we consider the  construction of  irreducible representations of finite pattern groups in terms of
Panov's associative polarization, which is a finite-field analogue of Kirillov's orbital method.
Using this construction, first, we are able to classify the irreducible representations of  the unipotent radical of the standard parabolic subgroups of $\GL_n$ with 4 parts;
second, we can parameterize   irreducible characters of degree $q$ in terms of coadjoint orbits of  cardinality $q^2$, for any finite pattern groups $G$ over $\BF_q,$ where  $\BF_q$ is a finite field with $q$ elements.
 
\end{abstract}
\maketitle{}

\section{Notation and Preliminary}
\label{sec:notation-preliminary}

Let $F$ be a self-dual field, that is, a finite field, or $\RR$, or $\CC$, or  a finite extension of $\QQ_p$,
or a field $\FF_q((t))$ of formal Laurent series in one variable over a finite field. 
We consider the type $A$ Lie group $\GL_{n}$ over $F$.
Let $S$ denote the diagonal torus of $\GL_{n}$, $U_{n}$  the upper triangular unipotent subgroup of $\GL_{n}$ and  $\Delta_{n}$  the set of positive roots:
\begin{equation*}
  \Delta_{n} = \{ \varepsilon_{i} - \varepsilon_{j} \spacedvert 1\le i < j \le n \}.
\end{equation*}

Let $\fg_{\alpha}$ denote the root space of $\gl_{n}$ and $U_{\alpha}$ the root subgroup of $\GL_{n}$ associated to $\alpha$.
Fix  a  group isomorphism
\begin{align*}
  x_{\alpha}: F \rightarrow U_{\alpha} 
\end{align*}
that is uniquely determined by the requirement $t x_{\alpha}(\xi) t^{-1}= x_{\alpha}(\alpha(t) \xi)$ for all $t\in S$ and $\xi\in F$.

We note the commutator relations.
Let $\alpha,\beta\in \Delta_{n}\cup -\Delta_{n}$ and assume that $\alpha + \beta\neq 0$.
Then for $\xi,\eta\in F$,  if $\alpha+\beta\in \Delta_{n}\cup -\Delta_{n}$, then
\begin{align*}
  x_{\alpha}(\xi) x_{\beta}(\eta)   x_{\alpha}(\xi)^{-1} x_{\beta}(\eta)^{-1} =  x_{\alpha+\beta} (\pm \xi\eta);
\end{align*}
otherwise $x_{\alpha}(\xi)$ and $x_{\beta}(\eta)$ commute.
More precisely, if $\alpha=\varepsilon_{i}-\varepsilon_{j}$ and $\beta=\varepsilon_{j}-\varepsilon_{k}$, then we take the $+$ sign and
if $\alpha=\varepsilon_{j}-\varepsilon_{k}$ and $\beta=\varepsilon_{i}-\varepsilon_{j}$, then we take the $-$ sign.

Let  $D \subseteq \Delta_{n}$.
Set  $\fg_{D} = \oplus_{\alpha\in D}  \fg_{\alpha}$. 

We say that a subset $D \subseteq \Delta_{n}$ is closed if for $\alpha,\beta\in D$ such that $\alpha + \beta \in \Delta_{n}$, then $\alpha + \beta \in D$.
Now assume that $D \subseteq \Delta_{n}$ is closed.
We say that $\alpha\in D$ is primitive, if it cannot be written as the sum of two roots in $D$.
Denote the set of primitive elements in $D$ by $D^{0}$ and denote the set of non-primitive elements in $D$ by $D^{\sharp}$.
By the commutator relations, $[\fg_{D},\fg_{D}]=\fg_{D^{\sharp}}$.
We call $\fg_{D}$ a pattern algebra.
Write $\fg_{D}^{*}$ for its $F$-linear dual.
Let $G_{D}$ be the subgroup of $\GL_{n}$ generated by $U_{\alpha}$ for $\alpha\in D$.
It is a subgroup of $U_{n}$.
We call it a pattern group. 
 Given a pattern algebra $\fg$ or a pattern group $G$, we  use $\Delta_{\fg}$ or $\Delta_{G}$ to indicate the closed set associated to $\fg $ or $G$. 

Next we discuss the coadjoint orbits which arise from the coadjoint action $\Ad^{*}$ of $G_{D}$ on $\fg_{D}^{*}$.
We write $\form{\ }{\ }$ for the natural paring between $\fg_{D}^{*}$ and $\fg_{D}$.
For $T\in \fg_{D}^{*}$ and $g\in G_{D}$, the coadjoint action of $G_{D}$ on $\fg_{D}^{*}$ is given by
\begin{align}\label{eq:coadj}
  \form{\Ad^{*}(g) T}{X} = \form{T}{\Ad(g^{-1})X} \quad\text{for all $X\in\fg_{D}$}.
\end{align}
We have a perfect pairing 
\begin{align*}
  \gl_{n} \times \gl_{n} &\rightarrow F\\
   (X,Y) &\mapsto \tr(XY).
\end{align*}
which restricts to a perfect pairing (where  $D$ is not necessarily closed)
\begin{align*}
  \fg_{D} \times \fg_{-D} &\rightarrow F.
\end{align*}
Thus we have an isomorphism $\fg_{D}^{*} \isom \fg_{-D}$.
As a result, sometimes we will write $\form{T}{X}$ for $\tr(TX)$ with $T\in\fg_{-D}$ and $X\in\fg_{D}$.

Now we make the coadjoint action of $G_{D}$ on $\fg_{-D}$ explicit.
Let $T\in\fg_{-D}$.
From~\eqref{eq:coadj}, we get
\begin{align*}
  \tr(\Ad^{*}(g) T X) =& \tr( T \Ad(g^{-1})X) \\
  =& \tr(Tg^{-1}X g) = \tr ( gTg^{-1}X) = \tr([gTg^{-1}]_{\fg_{-D}}X) ,
\end{align*}
where we let $[\cdot]_{\fg_{-D}}$ be the projection from $\gl_{n}$ to $\fg_{-D}$.
In conclusion, the coadjoint action of $G_{D}$ on $\fg_{-D}$ is given by
\begin{align*}
  \Ad^{*}(g) T = [gTg^{-1}]_{\fg_{-D}}.
\end{align*}

Next we write down the linear characters of $G_{D}$.
Fix a nontrivial additive character $\psi: F\rightarrow \CC^{\times}$.
The derived group $G_{D}^{\der}$ of $G_{D}$ is generated by $U_{\alpha}$ for $\alpha$ running over non-primitive roots in $D$
and it is  a pattern group with  Lie algebra given by $\fg_{D^{\sharp}}$.
The linear characters of $G_{D}$ can be identified with the linear characters of 
\begin{equation*}
  G_{D}/G_{D}^{\der} \isom \prod_{\alpha\in D^{0}} U_{\alpha} \isom \prod_{\alpha\in D^{0}} \fg_{\alpha} = \fg_{D^{0}}.
\end{equation*}
This isomorphism gives rise to  the projection map $G_{D} \rightarrow \fg_{D^{0}}$ which we denote by  $[\cdot]_{\fg_{D^{0}}}$.
Then $[g_{1}g_{2}]_{\fg_{D^{0}}} = [g_{1}]_{\fg_{D^{0}}}+[g_{2}]_{\fg_{D^{0}}}$, for $g_{1},g_{2} \in G_{D}$. 

Let $T\in \fg_{-\Delta_{n}}$.
Set $\psi_{T}(g) := \psi(\tr(T (g-I))) =  \psi(\tr(T g))$ for $g\in G_{D}$.
In general, it is not a linear character of $G_{D}$.
It is a linear character of $G_{D}$ exactly when $\form{T}{\fg_{D^{\sharp}}} = 0$ or equivalently $\tr T \fg_{D}^{2}=0$.
All linear characters of $G_{D}$ arise in this way.

 \section{Auxiliary Results}
In this section, we recall some notations and auxiliary results.

\begin{lem}\cite[Lemma 4.1.]{DI08}\label{sameno} Let $\fa$ be a finite dimensional associative nilpotent algebra over $\BF_q$.
Then the number of coadjoint orbits is the same as the number of conjugacy classes,
 which is also equal to the number of isomorphism classes of irreducible representations of $1+\fa$.
\end{lem}

Let $\Ind^{G} $ denotes the  induction functor.
Let $ \fa$ be a finite dimensional associative nilpotent algebra over $F$.
Then $\fa$ inherits a natural topology from $F$, and $1+\fa$ becomes a locally compact (Hausdorff) and
second countable topological group. Moreover, it is unimodular.
If $\fb$ is an $F$-subalgebra
of $\fa$, then $1+\fb$ can be viewed as a closed subgroup of $1+\fa.$
In this situation,
Mitya Boyarchenko \cite{Bo11} verifies a conjecture proposed by E. Gutkin \cite{Gutkin-MR323915} in 1973 by establishing the following theorem.

\begin{thm}\label{Bo}\cite[Theorem 1.3.]{Bo11}
  Let $\pi: 1+\fa\mapsto U(H)$ be a unitary irreducible representation. 
  Then	there exist an $F$-subalgebra $\fb\subset \fa$ and a unitary character $\alpha: 1 + \fb \mapsto  \CC^{\times}$ such that
	$\pi\cong \uInd_{1+\fb}^{1+\fa}\alpha,$ where $\uInd$ denotes the operation of unitary induction.
\end{thm}

Assume that the finite group  $G =M\ltimes N$  is a  semidirect product of $M$ by $N$, where $N$ is an abelian normal subgroup of $G$.  Let $m\in M$ act on $\widehat N$ by
$m\cdot \chi=\chi^m,$ where $\chi^m\in \widehat N$ is given by
$$\chi^m(n)=\chi(m^{-1}nm) \text{ for }n\in N.$$
Let $ R_\chi=\{m\in M\mid \chi^m=\chi\}$  be the stabilizer of $\chi$ in $M$.
Then the irreducible representations of $G$  can be described in terms of induced representations of the following form.

\begin{thm}\cite[Sec 8.2, Proposition 25.]{Se77}\label{Clifford} For
	$\chi, \ \chi'\in\widehat N$, and $\pi \in \widehat R_\chi,\ \pi' \in \widehat R_{\chi'},$  the followings hold.
	\begin{enumerate}
		\item
		$\Ind_{R_{\chi}\ltimes N}^G\pi \otimes\chi$ is an  irreducible representation of $G$.
		\item $\Ind_{R_{\chi}\ltimes N}^G\pi \otimes\chi\cong \Ind_{R_{\chi'}\ltimes N}^G\pi' \otimes\chi'$ if and only if $\chi$ and $\chi'$ are in the same $M$-orbit, and $\pi\cong \pi'$.
		\item Every irreducible representation $\pi$  of $G$ is isomorphic to
		$\Ind_{R_{\chi}\ltimes N}^G\pi \otimes\chi,$ for some $\chi   \in\widehat N$, and $\pi  \in \widehat R_\chi.$
	\end{enumerate}
	
\end{thm}

Now let $F=\FF_{q}$.
Let $U_n=1+\fu_n$ be the unipotent group, corresponding to  the closed set $\Delta_n.$ Let $\fg$ be a subalgebra of $\fu_n,$ and $G=1+\fg$.
Denote by $\mathfrak O_G$   the set of coadjoint orbits and by  $\mathfrak O_\lambda$  the coadjoint orbit containing $\lambda\in \fg^{*}$.
Let $\psi$ be any fixed nontrivial additive characters of $\BF_q$.
Note that if $\fg^\lambda$ is the stabilizer of $\lambda \in \fg^{*}$, then the subgroup $G^\lambda=1+\fg^\lambda$ is the stabilizer of $\lambda\in \fg^{*}$  and
\begin{equation}\label{dim eq}
\frac{\mid G\mid}{\mid G^\lambda\mid}=\frac{\mid \fg\mid}{\mid \fg^\lambda\mid}=q^{\dim \fg-\dim \fg^\lambda}=q^{\dim  \fO_\lambda} .
\end{equation}

 \begin{defn}
 Let $\fp$ be a subalgebra of $\fg.$ For  $\lambda\in \fg^*$, define a skew symmetric bilinear form   $$B_\lambda(x, y):=\lambda([x,y]),  \text{ for }x,\ y \in \fg,\text{ where }[x,y]=xy-yx.$$  The space $\fp$ is called {\bf isotropic} if $B_\lambda(x, y)=0,$ for all $x, \ y \in\fp$.
 The space $\fp$ is a {\bf polarization} of $\lambda\in \fg^*,$ if $\fp$ is a maximal isotropic subspace for the skew symmetric bilinear form $B_\lambda(x, y)$ on $\fg$. 
\end{defn}

 \begin{prop}\cite[Proposition 1.]{Pa14}
Any linear form $\lambda$ on a nilpotent Lie algebra $\fg$ has a polarization $\fp_{\lambda}.$
\end{prop}

 Note that any polarization contains the stabilizer $\fg^\lambda.$ 

  \begin{prop}\cite[Proposition 2.]{Pa14}\label{panov}
Let $\fp$ be a polarization of  $\lambda\in \fg^*,$ $P=1+\fp$, $j$ be the natural projection of $\fg^*$ onto $\fp^*$, and $L(\lambda):=j^{-1}j(\lambda).$ Then
\begin{enumerate}[(i)]
\item $\dim \fp=\frac{1}{2}(\dim \fg+\dim \fg^{\lambda}).$
\item $\mid L(\lambda) \mid =\sqrt{\mid \fO_\lambda\mid}.$
\item $L(\lambda)=\Ad^*_P \lambda.$ In particular, $ L(\lambda)\subset  \fO_\lambda.$
\end{enumerate}
\end{prop}
 When the exponential
map exists and  is a bijection from the Lie algebra $\fg$ onto the subgroup $G= \exp(\fg)$ of the
unitriangular group $U_n$, the analogue of classical  orbital method provides the classification of $\widehat G$ as follows.

Recall that $\psi: \BF_q\mapsto \bC^\times$ is a nontrivial additive character. 
Restriction of $\lambda$ on its polarization $\fp$ defines a character (one dimensional
 representation) $\eta_{\lambda}$  of the group $P = \exp(\fp)$ by the formula
 $$\eta_{\lambda}(\exp(x)) = \psi(\lambda(x)).$$

 \begin{thm}\cite[Theorem 2.]{Pa14}\label{orbit} Let
 $T^\lambda :=\Ind_P^G \eta_{\lambda}.$  Then 
\begin{enumerate}
\item $\dim T^\lambda=q^{\frac{1}{2}\dim \fO_\lambda} = \sqrt{|\fO_\lambda |}.$
\item The representation $T^\lambda$ does not depend on the choice of polarization.
\item The representation $T^\lambda$ is irreducible.
\item The representations $T^\lambda$ and $T^{\lambda'}$ are equivalent if and only if $\lambda$ and $\lambda'$ belong to the same coadjoint orbit.
\item For any $\pi\in \widehat G,$ there exists $\lambda \in \fg^*$ such that $\pi$ is equivalent to $T^\lambda.$ 

\end{enumerate}

\end{thm}

Let $p$ be   the characteristic of $\BF_q.$ 
When $p>n,$  the exponential
map is a bijection from the Lie algebra $\fu_n(q)$ onto  $U_n(q)$,
and the classification of $\widehat U_n(q)$ is established by Theorem \ref{orbit}.
 
On the other hands, when the exponential
map is not applicable, the orbital method may be modified as follows.

\begin{defn}\label{asspolarization}
An associative polarization of $\lambda\in  \fg^{*}$ is a polarization $\fp$ satisfying the following conditions:
\begin{enumerate}
\item $\fp$ is an associative subalgebra of $\fg$.
\item $\lambda(\fp^2)=0$.
 \end{enumerate}
\end{defn}

 When $\lambda\in  \fg^{*}$ admits an associative polarization $\fp$, $\psi_\lambda (1+x):=\psi( \lambda(x))$ is a representation of $P=1+\fp.$ Similarly,  let
 $$\eta_{\lambda}(1+x) = \psi\circ\lambda(x).$$
 Consider the induced representation
 $M^\lambda :=\Ind_P^G \eta_{\lambda}.$

 \begin{thm}\cite[proposition 2.3.]{Pa15}\label{2.3}
\begin{enumerate}
\item $\dim M^\lambda=q^{\codim  \fp} = \sqrt{|\fO_\lambda |}, $ where $\codim \fp$ denotes the co-dimension of $\fp $ in $\fg$.
\item The representation $M^\lambda$ does not depend on the choice of the associative polarization.
\item The representation $M^\lambda$ is irreducible.
\item Assume that $\lambda$ and $\lambda'$ both admit associative polarizations.
Then the representations $M^\lambda$ and $M^{\lambda'}$ are equivalent if and only if $\lambda$ and $\lambda'$ belong to the same coadjoint orbit.
\end{enumerate}
\end{thm}

 \section{Finite Groups of Good Types}

\begin{defn}\label{goodtype}
Let $G=1+\fg $ be a finite pattern group over $\BF_q$ with associated nilpotent algebra $\fg.$   
For  $S\in   \fg^t,$ let $\lambda_S\in \fg^*$ denote  the map $$\lambda_S: \fg\mapsto \BF_q \text{ given by }\lambda_S(x)=\tr (SX).$$ 
Then we call  $G$ is of {\bf good type,} if an associative polarization exists for any $\lambda_S,$ where  $S\in   \fg^t.$
 \end{defn}

 If $G$ is of good type, then by Lemma~\ref{sameno} and Theorem~\ref{2.3}, $$\widehat G=\{\Ind_{1+\fb_S}^{G}\psi_S\mid S\in \fO_G\},$$ where $\psi_S(1+x)=\psi(\tr Sx).$

In  \cite{Ni20}, \cite{Ni21} \cite{LN23}, and \cite{Ni24}, several series of finite pattern groups  are verified to be of good type, besides those pattern groups $G = 1+\fg$ admitting bijective exponential map between $\fg$ and $G$.   However, \cite{Ev11} provides examples of finite pattern groups which are not of good type.

\subsection{Unipotent radical of 4 parts}   
Let $U_{n_1, n_2, n_3, n_4}$ denote the unipotent radical of the standard parabolic subgroups $P_{n_1, n_2, n_3, n_4}$ of $U_n$, corresponding to the partition of $n=n_1+n_2+n_3+n_4,\ n_i\in \BN.$

\begin{thm}\label{4parts}
  Let $G:=U_{n_1, n_2, n_3, n_4}.$ Then $G$ is of good type. 
\end{thm}
\begin{proof}
By  Theorem \ref{2.3}, to show that $G =1+\fg $ is of good type, it suffices to show that there exists  a subalgebra $\fb_T$ of $\fg $, satisfying  
$\tr T\fb_T^2=0$ and  $\codim \fb_T=\frac{\dim\mathfrak O_T}{2}, $  for any $T  \in  \fg^t.$

Let $T \in \fg^{t}$ with blocks given by $T_{ij}$ with $T_{ij} = 0$ for $i\le j$.
We find the dimension of $\Stab_{\fg}(T)$, namely, the set $\{ X\in \fg \spacedvert [[X,T]]_{\fg^{t}} = 0 \}$.
Recall that $[\cdot]_{\fg^{t}}$ is the projection of $\gl_{n}$ to $\fg^{t}$.

Let $X \in \fg$.
First we note that
\begin{align*}
  \Ad^{*}(1+X)\cdot T - T  = [[X,T] (1+X)^{-1}]_{\fg^{t}} = [[[X,T]]_{\fg^{t}} (1+X)^{-1}]_{\fg^{t}}.
\end{align*}
Explicitly, we have
\begin{align*}
[[X,T]]_{\fg^{t}} =  \begin{pmatrix}
    0 & 0 & 0 & 0 \\
X_{23} T_{31} + X_{24} T_{41} & 0 & 0 & 0 \\
X_{34} T_{41} & -T_{31} X_{12} + X_{34} T_{42} & 0 & 0 \\
0 & -T_{41} X_{12} & -T_{41} X_{13} - T_{42} X_{23} & 0
  \end{pmatrix}.
\end{align*}
Thus $\Stab_{\fg}T$ consists of $X \in \fg$ satisfying
\begin{align*}
 X_{23} T_{31} + X_{24} T_{41} &= 0,\\
 X_{34} T_{41} &= 0,\\
 -T_{31} X_{12} + X_{34} T_{42} &= 0,\\
 -T_{41} X_{12} &= 0 ,\\
 -T_{41} X_{13} - T_{42} X_{23} &= 0  .
\end{align*}

The expression of $[[X,T]]_{\fg^{t}}$ shows that we may choose another $T$ in the same coadjoint orbit such that  $\spanrow(T_{31}) \cap \spanrow(T_{41}) = 0$ and $\spancol(T_{42}) \cap \spancol(T_{41}) = 0$, where we use $\spanrow$ (resp. $\spancol$) to denote the row (resp. column) space of the matrix.
Then for such a $T$,
$\Stab_{\fg}T$ consists of $X$ satisfying
\begin{align*}
 X_{23} T_{31}                   &= 0,   \\
 X_{24} T_{41}                   &=0,    \\
 X_{34} T_{41}                   &= 0,   \\
 -T_{31} X_{12} + X_{34} T_{42}  &= 0,   \\
 T_{41} X_{12}                  &= 0 ,  \\
 T_{41} X_{13}                  &= 0  , \\
 T_{42} X_{23} &=0.    
\end{align*}
The constraint on $X_{13}$ contributes 
\begin{equation*}
  \rank(T_{41})n_{3}
\end{equation*}
to the codimension.
The constraint on $X_{24}$ contributes 
\begin{equation*}
  \rank(T_{41})n_{2}
\end{equation*}
to the codimension.
By Lemma~\ref{lemma:codim}\eqref{item:codim-1}, the constraints on $X_{23}$ contribute 
\begin{equation*}
  \rank(T_{31})n_{2} + \rank(T_{42})n_{3} - \rank(T_{31})\rank(T_{42})
\end{equation*}
to the codimension.
There are three equations giving constraints on $X_{12}$ and $X_{34}$, which by Lemma~\ref{lemma:codim}\eqref{item:codim-2} contribute 
  \begin{multline*}
    \rank(T_{41})n_{2} + \rank(T_{41})n_{3} \\
    + \rank(T_{31})\rank(T_{42}) + (n_{2} - \rank(T_{42})) \rank(T_{31}) + (n_{3} -\rank(T_{31})) \rank(T_{42})
  \end{multline*}
to the codimension.
Thus the codimension of $\Stab_{\fg}T$ in $\fg$ is 
\begin{equation}\label{eq:codim-stab}
  2( n_{3} \rank(T_{41}) + n_{2} \rank(T_{41}) + n_{2} \rank(T_{31}) + n_{3} \rank(T_{42}) - \rank(T_{31}) \rank(T_{42})).
\end{equation}

Next we exhibit a subalgebra $\fb_{T}$ of $\fg$ of codimension equal to half of \eqref{eq:codim-stab} such that $\tr T\fb_T^2=0$.
For $X,Y \in \fg$, we have
\begin{align*}
  XY =
  \begin{pmatrix}
    0 & 0 & X_{12} Y_{23} & X_{12} Y_{24} + X_{13} Y_{34} \\
0 & 0 & 0 & X_{23} Y_{34} \\
0 & 0 & 0 & 0 \\
0 & 0 & 0 & 0
  \end{pmatrix}.
\end{align*}
Let $\fb_{T}$ be the subalgebra of $\fg$ consisting of  $Y$ satisfying
\begin{align*}
  Y_{23} T_{31}  &= 0,\\
  T_{42} Y_{23} &= 0,\\
  T_{41} Y_{12}  &= 0,\\
  Y_{34} T_{41} &= 0.
\end{align*}
With these conditions in place, we have $\tr TXY =0$ for all $X,Y\in \fb_{T}$ or  in other words $\tr T\fb_T^2=0$.
The contribution to the codimension by the constraints on $Y_{12}$ and $Y_{34}$ is $\rank(T_{41})(n_{2}+n_{3})$ and
that for $Y_{23}$ is $n_{2}\rank(T_{31}) + n_{3}\rank(T_{42}) - \rank(T_{31}) \rank(T_{42})$ by Lemma~\ref{lemma:codim}\eqref{item:codim-1}.
Thus the codimension of $\fb_{T}$ in $\fg$ is $(n_{2} + n_{3})\rank(T_{41}) + n_{2}\rank(T_{31}) + n_{3}\rank(T_{42}) - \rank(T_{31}) \rank(T_{42})$ which is exactly half of \eqref{eq:codim-stab}.

\end{proof}
 
\begin{lem}\label{lemma:codim}
  \begin{enumerate}
  \item \label{item:codim-1}The codimension of the subspace 
\begin{equation*}
  \{ X_{23} \in \mathrm{Mat}_{n_{2},n_{3}} \spacedvert T_{42}X_{23} = 0 , X_{23}T_{31} = 0 \}
\end{equation*}
of $\mathrm{Mat}_{n_{2},n_{3}}$
 is 
\begin{equation*}
  n_{3} \rank(T_{42})  + n_{2} \rank(T_{31}) - \rank(T_{31})\rank(T_{42}),
\end{equation*}
where $T_{31},T_{42}$ are suitable matrices such that $T_{42}X_{23} $ and $X_{23}T_{31} $ are defined.
\item \label{item:codim-2}
 The codimension of the subspace 
\begin{equation*}
  \{ X_{12}\in \mathrm{Mat}_{n_{1},n_{2}} , X_{34} \in \mathrm{Mat}_{n_{3},n_{4}}  \spacedvert T_{31}X_{12} = X_{34}T_{42}, T_{41}X_{12} = 0, X_{34}T_{41} = 0 \}
\end{equation*}
of $\mathrm{Mat}_{n_{1},n_{2}} \times \mathrm{Mat}_{n_{3},n_{4}} $  is
  \begin{multline*}
    \rank(T_{41})n_{2} + \rank(T_{41})n_{3} \\
    + \rank(T_{31})\rank(T_{42}) + (n_{2} - \rank(T_{42})) \rank(T_{31}) + (n_{3} -\rank(T_{31})) \rank(T_{42})
  \end{multline*}
where $T_{31},T_{42},T_{41}$ are suitable matrices such that $T_{31}X_{12}$, $X_{34}T_{42}$, $T_{41}X_{12}$ and $X_{34}T_{41}$ are defined and $\spanrow(T_{31}) \cap \spanrow(T_{41}) = 0$ and $\spancol(T_{42}) \cap \spancol(T_{41}) = 0$.
\end{enumerate}
\end{lem}

\begin{proof}
  Consider the first part.
  Let $r_{42} = \rank(T_{42})$ and $r_{31} = \rank(T_{31})$.
  Let $P_{1},P_{2},P_{3},P_{4}$ be invertible matrices such that
  \begin{align*}
    P_{4}T_{42}P_{2} = \begin{pmatrix}
I_{r_{42}}  & 0 \\
0           & 0
    \end{pmatrix} =: U_{42} , \quad
        P_{3}T_{31}P_{1} = \begin{pmatrix}
I_{r_{31}}  & 0 \\
0           & 0
    \end{pmatrix} =: U_{31}.
  \end{align*}
  Then the equations defining the subspace are equivalent to
  \begin{align*}
    U_{42} P_{2}^{-1}X_{23} P_{3}^{-1} &= 0\\
    P_{2}^{-1} X_{23} P_{3}^{-1} U_{31} &= 0.
  \end{align*}
  Renaming $P_{2}^{-1}X_{23} P_{3}^{-1}$ to $X_{23}$.
  It then suffices to find the codimension of the subspace of  $X_{23} \in \mathrm{\Mat}_{n_{2},n_{3}}$ defined by
  \begin{align*}
    U_{42} X_{23} &= 0\\
    X_{23} U_{31} &= 0.
  \end{align*}
  It is elementary to see that the codimension is given by
  \begin{align*}
    n_{3} r_{42} + n_{2}r_{31} -r_{31}r_{42}
  \end{align*}
  as desired.

  Consider the second part.
    Let $r_{42} = \rank(T_{42})$, $r_{31} = \rank(T_{31})$ and $r_{41} =\rank(T_{41})$.
  Let $P_{4}, P_{1}$ be invertible matrices such that $P_{4}T_{41}P_{1} = \begin{pmatrix}
I_{r_{41}}  & 0 \\
0           & 0
  \end{pmatrix} =: U_{41} $.
  Renaming $P_{1}^{-1}X_{12}$ to $X_{12}$ and $X_{34}P_{4}^{-1}$ to $X_{34}$, we may assume the equations are given by
  \begin{align*}
    T_{31} X_{12} &= X_{34} T_{42},\\
    U_{41} X_{12} &= 0,\\
    X_{34} U_{41} &= 0, 
  \end{align*}
  where we have $\spanrow(T_{31}) \cap \spanrow(U_{41}) = 0$ and $\spancol(T_{42}) \cap \spancol(U_{41}) = 0$.
  Then we may further require that $T_{31} = (0\ \  S_{31})$ for some $S_{31}$ of size $n_{3} \times (n_{1} - r_{41})$ and of rank $r_{31}$ and
  $T_{42} = \begin{pmatrix}
    0 \\ S_{42}
  \end{pmatrix}$ for some $S_{42}$ of size $(n_{4}-r_{42})\times n_{2}$ and of rank $r_{42}$.

  Write $X_{12} = \begin{pmatrix}
    X_{12}^{(1)} \\  X_{12}^{(2)}
  \end{pmatrix}$ where $X_{12}^{(1)}$ has $r_{41}$ rows and $X_{12}^{(2)}$ has $n_{1}-r_{41}$ rows.
  Write $X_{34} =     ( X_{34}^{(1)} \ \  X_{34}^{(2)}) $ where $X_{34}^{(1)}$ has $r_{41}$ columns and $X_{34}^{(2)}$ has $n_{4}-r_{41}$ columns.
  Then the equalities $U_{41} X_{12} = 0$ and $X_{34}U_{41}$ shows that $X_{12}^{(1)} = 0$ and $X_{34}^{(1)}=0$, which gives the contribution of $r_{41} n_{2} + r_{41} n_{3}$ to the codimension.

  It remains to analyse the equations
  \begin{equation}\label{eq:codim}
 S_{31} X_{12}^{(2)} = X_{34}^{(2)} S_{42}.
  \end{equation}
  Let $Q_{1},Q_{2},Q_{3},Q_{4}$ be invertible matrices such that
  \begin{align*}
    Q_{3} S_{31} Q_{1} = \begin{pmatrix}
I_{r_{31}}  & 0 \\
0           & 0
    \end{pmatrix} =: U_{31} , \quad
        Q_{4}S_{42}Q_{2} = \begin{pmatrix}
I_{r_{42}}  & 0 \\
0           & 0
    \end{pmatrix} =: U_{42}.
  \end{align*}
  Then \eqref{eq:codim} is equivalent to
  \begin{align*}
 U_{31} Q_{1}^{-1} X_{12}^{(2)} Q_{2} &= Q_{3} X_{34}^{(2)} Q_{4}^{-1} U_{42}.    
  \end{align*}

Renaming $Q_{1}^{-1} X_{12}^{(2)} Q_{2}$ to $X_{12}^{(2)}$ and $Q_{3} X_{34}^{(2)} Q_{4}^{-1} $ to $X_{34}^{(2)}$, we get the equations
\begin{align*}
  U_{31}  X_{12}^{(2)}  &=  X_{34}^{(2)}U_{42} ,
\end{align*}
which contributes $r_{31} r_{42} + r_{31}(n_{2} - r_{42}) + r_{42} (n_{3}-r_{31})$ to the codimension.
Adding up all the contributions, we find the codimension as desired.
\end{proof}

 \begin{defn}
Given a pattern group $G=1+\fa$, if $\Delta_G\subset \Delta_n$ and $\Delta_G\not\subset \Delta_{n-1},$
then we say that $G$ (resp. $\fa$) has u-rank n.
\end{defn}

 \begin{defn}Given a pattern group $G=1+\fa$,
for any  $0\ne T\in \fa^t,$ if there exists a subalgebra $\fb_T$ of $\fa,$  with the same u-rank as $\fa,$ satisfying $\tr T(\fb_T)^2=0, $  
then we call  $(T, \fb_T)$ an inducible pair for $G$.
\end{defn}
When we have an  inducible pair  $(T, \fb_T)$, $\psi_T(1+x):=\psi(\tr Tx)$ is a well-defined character of $1+\fb_T$.

\begin{thm}\label{inducible} Assume that $G=1+\fa$ has u-rank $n\ge 2.$
For  any $0\ne T\in \fa^t,$ there exists $\fb_T$ such that  $(T, \fb_T)$ is an inducible pair.
\end{thm}
\begin{proof}

We apply induction on  the u-rank  $n$.
When $n=2,$  for any $T\in \fa^t,$ we may choose $\fb_T=\fa$. Then $\tr T(\fb_T)^2=0,$ since $\fa^2=0.$ 
Assume that the statement holds for any pattern group with u-rank $n\le k.$

For $1+\fa$ with u-rank $n=k+1$ and $T=(T_{i,j})\in \fa^t,$ by choosing suitable representative in the same coadjoint orbit, 
we may assume that $T_{n,i}=0,$ if both $(j,i), (i, n)\in \Delta_\fa $ and $T_{n,j}\ne 0,$  
for some $j<i$.
Let $T'$ be the restriction of $T$ to $M=1+\fm$, where $G=M\ltimes Z$, where $Z=1+\fz,$ and  
$\Delta_{\fz}=\{(i, n)\mid (i, n)\in \Delta_\fa\}$ and $\Delta_{\fm}=\{(i, j)\mid (i, j)\in \Delta_\fa, j<n\}.$

By induction assumption, let $\fc$ be a subalgebra of $\fm$  such that $(T', \fc)$ is an inducible pair. 
Let $$\fn_1=\{(n_1, \cdots, n_k)\mid n_i=0 \text{ if }(j,i), (i,n) \in \Delta_\fa\text{ and }T_{n,j}\ne 0\}.$$
Let $\fb=\begin{pmatrix}\fc&\fn^t\\
0&0\end{pmatrix}\subset \fa,$  where $ \fn_1\supset \fn$ and $$\fn:=\{(n_1, \cdots, n_k)\mid n_r=0 \text{ if for some }1\le j\le k, 
(j,r)\in \Delta_\fm,\text{ and }(r,n)\in \Delta_\fz\setminus\Delta_{\fn_1^{t}}\}.$$
 Then $(T, \fb)$ is an inducible pair. 
 
\end{proof}

We recall 2 earlier results  about inducible pairs.

 \begin{thm}\cite[Theorem 3.1]{LN23}\label{monomial}
For any finite pattern group $G=1+\fa$ with pattern algebra $\fa$, given any monomial matrix  $T=(t_{i,j})\in \fa^t$,
 there exists a pattern subgroup $H_T$ of $G$,  associated to the closed set 
$ \Delta_{H_T} $ defined inductively, such that
  $\Ind_{H_T}^{G}\psi_T$ is an irreducible representation of $G$, where $\psi_T(x)=:\psi(trTx)$ for $x\in H_T$.
Moreover, for monomial $T,\ S \in \fa^t,$ 
$$\Ind_{H_T}^{G}\psi_T \cong\Ind_{H_S}^{G}\psi_S\text{ if and only if }T=S.$$ 	
\end{thm}

\begin{thm}\cite[Theorem 2.2]{Ni24}\label{degreeq} 
Let $G=1+\fa$ be a pattern group with pattern algebra $\fa$ over the finite field $\BF_q$. Then every degree q irreducible character of $G$ is isomorphic to  $$\Ind_{1+\fb}^{G }\psi_T, \text{ for some }T\in \fa^t , 
$$
where $\fb$ is a pattern subalgebra of $\fa$ satisfying $\tr T\fb^2=0.$  
Moreover, if $\Ind_{1+\fb }^{G }\psi_S\cong \Ind_{1+\fc }^{G }\psi_{T},$ then
$S$ and $T$ are in the same coadjoint orbit.
\end{thm}

For $G$ a Sylow $p$-subgroup of a classical group defined over a finite
 field of characteristic $p>2$ and $\fg$ its associated Lie algebra, in \cite{Sa09} Sangroniz, Josu
 showed that the irreducible characters of $G$ with sufficiently large
 degree $p^f$ can be parameterized by the coadjoint orbits of $G$ of size $p^{2f}$. Here we prove a similar result for degree $q$ characters of finite pattern groups.

\begin{thm}
Let $G=1+\fa$ be a pattern group with pattern algebra $\fa$ over the finite filed $\BF_q$. Then the number of  irreducible characters of degree $q$ equals to the number of coadjoint orbits with cardinality $q^2$.
\end{thm}
\begin{proof}
Let $D_q$ denote the set of  irreducible characters of degree $q$ and $C_{q^2}$ the set of coadjoint orbits with cardinality $q^2.$
First, we show that for any coadjoint orbit in $X\in \fO_X \in C_{q^2}$ we can choose a representative $Y\in \fb^t \cap \fO_X, $ for some pattern subalgebra $\fb$ of $\fa$ with codimension $1$. Let $X=(x_{i,j})\in \fa^t.$

 Since $\mid \fO_X \mid=q^2,$ there are at most 2 nonzero entries $x_{t,s}, x_{k, m},$ with $s<k, $ such that $(s,t), (m,k)\in \Delta_{\fa^2}$. Then  $t <m$  and $(s, k), (k, t), (t, m)\in \Delta_{\fa}.$
 Let $\fb$ be the subalgebra of $\fa$ such that $\Delta_{\fb}=\Delta_{\fa}\setminus \{(k,t)\}.$ Then we may choose a representative 
$Y=(y_{i,j})\in \fb^t \cap \fO_X, $ with $y_{t,k}=0.$ Then $Y\in (\fb\setminus\fb^2)^t$ and $\Ind_{1+\fb}^{1+\fa}\psi_Y$ is irreducible by Theorem \ref{2.3}.

 If there is exactly 1 nonzero entry $x_{t,s}$  with $(s,t)\in \Delta_{\fa^2}$, then there exists a unique $r,$ with $s<r<t$ such that $(s,r), (r,t)\in \Delta_{\fa}.$   
  Let $\fb$ be the subalgebra of $\fa$ such that $\Delta_{\fb}=\Delta_{\fa}\setminus \{(s,r)\}.$ Then we may choose a representative 
$Y=(y_{i,j})\in \fb^t \cap \fO_X, $ with $y_{r,s}=0.$ Then $Y\in (\fb\setminus\fb^2)^t$ and $\Ind_{1+\fb}^{1+\fa}\psi_Y$ is irreducible by Theorem \ref{2.3}.

On the other hand, if $\pi$ is an irreducible character of degree $q$, then $\pi\cong\Ind_{1+\fb }^{G }\psi_T,$ for some $T\in \fb^t$ and $\fb$ is a subalgebra of $\fa$ with codimension $1$. Then $\mid \fO_T \mid=q^2.$ Hence 
 the number of  irreducible characters of degree $q$ equals the number of coadjoint orbits with cardinality $q^2.$
\end{proof}

\end{document}